\documentclass[12pt]{amsart}
\usepackage[utf8x]{inputenc}
\usepackage{amssymb}
\usepackage{amscd}
\usepackage{amsmath}
\usepackage{enumerate}
\usepackage[all]{xy}
\usepackage{verbatim}
\usepackage{tikz}
\usepackage{xcolor}
\usepackage{todonotes}
\usepackage[pdftex,backref]{hyperref}

\newtheorem{theo}{Theorem}[section]

\newtheorem{cor}[theo]{Corollary}
\newtheorem{prop}[theo]{Proposition}

\newtheorem{definition}[theo]{Definition}

\newtheorem{remark}[theo]{Remark}

\newtheorem{example}[theo]{Example}

\def\e{{\epsilon}}
\def\a{{\alpha}}
\def\b{{\beta}}
\newcommand{\B}{{\mathbb{B}}}
\newcommand{\C}{{\mathbb{C}}}
\newcommand{\CC}{{\mathcal{C}}}
\newcommand{\R}{{\mathbb{R}}}

\newcommand{\BS}{\mathbb{S}}
\newcommand{\BD}{\mathbb{D}}
\newcommand{\D}{\mathbb{D}}

\begin{document}
\title[]{Milnor-L\^e type fibrations for subanalytic maps}

\author{Rafaella de Souza Martins}
\author{Aur\'elio Menegon}

\date{31-05-2018}
\address{Rafaella de Souza Martins: Universidade de S\~ao Paulo - Brazil.}
\email{rafaella.souza.martins@usp.br}
\address{Aur\'elio Menegon: Universidade Federal da Para\'iba - Brazil.}
\email{aurelio@mat.ufpb.br}

\begin{abstract}
We prove a Milnor-L\^e type fibration theorem for a subanalytic map $f: X \to Y$ between subanalytic sets $X \subset \R^m$ and $Y \subset \R^n$. Moreover, if $f$ extends to an analytic map $\R^m \to \R^n$, we define the singular set and the discriminant set of $f$, in a stratified sense. Then we give Milnor-L\^e type fibration theorems for $f$ outside its discriminant, as well as over it. We give examples of each situation approached in this paper.
\end{abstract}
\maketitle

\medskip
\section{Introduction}  
\label{section_1}
\medskip

The topological behavior of the non-critical levels of an analytic map near a critical point has been studied by many authors. It probably started with Milnor's work in \cite{Mi}, where he showed that if $f: (\C^n,0) \to (\C,0)$ is a holomorphic function-germ then there exist $\e$ and $\eta$ sufficiently small, with $0<\eta \ll \e$, such that the restriction:
$$f_|: f^{-1}(\BD_\eta \backslash \{0\}) \cap \B_\e \to \BD_\eta \backslash \{0\}$$
is the projection of a smooth locally trivial fibration, where $\B_\e$ denotes the closed ball of radius $\e$ around $0 \in \C^n$ and $\BD_\eta$ denotes the closed disk of radius $\eta$ around $0 \in \C$. 

Milnor then described the topology of the fiber $F_f$ of such fibration, which is now called the Milnor fiber of $f$ at $0$. One can obtain from $F_f$ some important invariants of the germ of the hypersurface $V_f = f^{-1}(0)$ at $0 \in \C^n$, like the Milnor number if $f$ has an isolated critical point, or the L\^e numbers (\cite{Ma}) in the general case. Latter, H. Hamm (\cite{Ha}) and D.T. L\^e (\cite{Le}) proved a fibration theorem as above for a complex function-germ defined in a complex analytic set $X$.

In the real setting, there is also a Milnor-L\^e type fibration theorem for map-germs $f: (\R^m,0) \to (\R^n,0)$ of class $C^1$, provided that $f$ satisfies some transversality condition (see \cite{CMSS}). In this case, there is a differentiable locally trivial fibration:
$$f_|: f^{-1}(\BD_\eta \backslash \Delta_f) \cap \B_\e \to (\BD_\eta \backslash \Delta_f) \cap f(\R^m) \, ,$$
where $\Delta_f$ is the discriminant set of $f$, which is defined as the image of the critical set of $f$. 

This leads us to the question of what happens over the discriminant set of $f$. What does the topological behavior of the critical levels of $f$ tell about $V(f)$? The first step in this direction is to give a fibration theorem for $f$ over $\Delta_f$. This is what we envisage in this paper.

The main difficulty lies in the fact that the discriminant $\Delta_f$, in general, is not an analytic set. So we must enter into the realm of the subanalytic sets. In Section 2, we briefly recall some basic definitions and results about subanalytic sets and Whitney stratifications. 

In Section 3, we consider a subanalytic map $f: X \to Y$ between subanalytic sets $X \subset \R^m$ and $Y \subset \R^n$, with $m \geq n$, such that $f$ extends to a continuous map $\tilde{f}: \R^m \to \R^n$. We prove:

\begin{theo} \label{theo_1}
Let $f: X \to Y$ be a subanalytic map as above. Let $B \subset \R^m$ be a compact subanalytic set and let $Z \subset \R^n$ be a subanalytic set contained in $f(X \cap B)$, with $\dim Z>0$. Then there exist subsets $W_0, \dots, W_l$, with $1 \leq l \leq \dim Z$, such that: 
\begin{itemize} 
\item[$(i)$] $W_0 \subset W_1 \subset \dots \subset W_{l-1} \subset W_l = Z$;
\item[$(ii)$] $W_i$ is subanalytic in $\R^n$, for each $i=0, \dots, l$;
\item[$(iii)$] $\dim W_{i-1} < \dim W_i$, for each $i=1, \dots, l$;
\item[$(iv)$] $W_i \setminus W_{i-1}$ is a smooth subset of $\R^n$ and subanalytic in $\R^n$, for each $i=1, \dots, l$;
\item[$(v)$] The restriction:
$$f_|: f^{-1}(W_i \backslash W_{i-1}) \cap B \to W_i \backslash W_{i-1}$$
is the projection of a topological locally trivial fibration, for each $i=1, \dots, l$. 
\end{itemize}
\end{theo}

%
%

In Section 4, we consider a subanalytic map $f: X \to Y$ between subanalytic sets $X \subset \R^m$ and $Y \subset \R^n$, with $m \geq n$, such that $f$ extends to an analytic map $\tilde{f}: \R^m \to \R^n$. We define the singular set and the discriminant set of $f$ in a stratified sense (Definition \ref{definition_1}). Then we give a stratified Milnor-L\^e type fibration theorem for $f$ outside its discriminant (Theorem \ref{theo_2}), as well as a Milnor-L\^e type fibration theorem for $f$ over its discriminant (Theorem \ref{theo_3}). In the particular case when $Y$ is the Euclidian space $\R^n$, we get:

\begin{theo} \label{theo_main}
Let $X \subset \R^m$ be a subanalytic set and let $f: X \to \R^n$ be a subanalytic map that extends to an analytic map $\tilde{f}: \R^m \to \R^n$. If $B \subset \R^m$ is a compact subanalytic set, then:
\begin{itemize}
\item[$(i)$] The restriction:
$$f_|: f^{-1}(\R^n \setminus \Delta_{f_B}) \cap B  \to  (\R^n \setminus \Delta_{f_B}) \cap f(X \cap B) $$
is the projection of a topological locally trivial fibration, where $\Delta_{f_B}$ is the discriminant set of $f$ in the stratified sense (see Definition \ref{definition_1}). \item[$(ii)$] Moreover, if $\Delta_{f_B}$ has positive dimension then there exists a subset $W \subset \Delta_{f_B}$, subanalytic in $\R^n$, with $\dim W < \dim \Delta_{f_B}$, such that $\Delta_{f_B} \backslash W$ is a smooth subset of $\R^n$, subanalytic in $\R^n$, and such that the restriction:
$$f_|: f^{-1}(\Delta_{f_B} \backslash W) \cap B \to \Delta_{f_B} \backslash W$$
is the projection of a topological locally trivial fibration.
\end{itemize}
\end{theo}

We observe that $(i)$ of Theorem \ref{theo_main} above generalizes Theorem 2.1 of \cite{CSG}, which concerns the case when $f: (\R^m,0) \to (\R^n,0)$ is an analytic map-germ and $B$ is a closed ball around $0$ in $\R^m$.

We give some examples of each situation approached in this paper.

\medskip
\section{Background}  
\medskip

Following \cite{Milman-Bierstone}, \cite{Shiota} and \cite{Verdier}, we recall some definitions and results about subanalytic sets. 

Let $M$ be a real analytic manifold. A subset $X$ of $M$ is {\it semianalytic} (in $M$) if each point $p \in M$ has a neighborhood $U_p$ in $M$ and real-valued functions $f_{jk}$ and $g_{jk}$ analytic on $U_p$ such that:
$$X \cap U_p = \bigcup_{j=1}^{\ell} \left( \bigcap_{k=1}^{s_j}\{x \in M; f_{jk}(x) = 0 \ \text{and} \ g_{jk}(x)>0\} \right) \, .$$
Clearly, every analytic set is a semianalytic set. 

A subset $X$ of $M$ is {\it subanalytic} (in $M$) if each point of $p \in M$ admits a neighborhood $U_p$ in $M$ such that $X \cap U_p$ is a projection of a relatively compact semianalytic set (i.e., there is a real analytic manifold $N$ and a relatively compact semianalytic subset $A$ of $M \times N$ such that $X \cap U_p = \pi(A)$, where $\pi:M \times N \to M$ is the projection).

Equivalently, a subset $X$ of $M$ is subanalytic if each point of $p \in M$ admits a neighborhood $U_p$ such that: 
$$X \cap U_p = \bigcup_{j=1}^n (f_{j1}(A_{j1})- f_{j2}(A_{j2})) \, ,$$
where, for each $j=1,\dots,n$ and $k = 1,2$, $A_{jk}$ is a closed analytic subset of a real analytic manifold $N_{jk}$, $f_{jk}:N_{jk} \longrightarrow U_p$ is real analytic, and $f_{{jk}_{|}}:A_{jk} \longrightarrow U_p$ is proper.

Given subsets $X \subset Y \subset M$, we say that $X$ is {\it subanalytic in $Y$} if each point of $p \in Y$ admits a neighborhood $U_p$ in $M$ such that $X \cap U_p$ is a projection of a relatively compact semianalytic set. Notice that an analytic set contained in $\R^n$ is not necessarily subanalytic in $\R^n$ because an analytic set is the zero set of an analytic function locally at each point of the set. For instance, take $X \subset \R$ given by the sequence of points $x_n = 1/n$, with $n=1, 2, \dots$, which is analytic but is not subanalytic in $\R$.

Important examples of subanalytic sets are a semianalytic set in $\R^n$ and a polyhedron imbedded and closed in $R^n$.

The intersection and union of a finite collection of subanalytic sets are subanalytic. Every connected component of a subanalytic set is subanalytic. The family of connected components is locally finite. A subanalytic set is locally connected. The closure of a subanalytic set is subanalytic. The complement (and thus the interior) of a subanalytic set is subanalytic.

Let $X \subset M$ be a subanalytic set and let $N$ be a real analytic mani\-fold. We say that $f: X \to N$ is a {\it subanalytic map} if its graph is subanalytic in $M \times N$. So the image of a relatively compact sub\-analytic set by a subanalytic mapping is subanalytic. An example of a subanalytic map is a PL map between polyhedra imbedded and closed in Euclidean spaces.

We have:

\begin{prop}\label{propriedades}(\cite{Shiota}, Prop. I.2.1.1 and I.2.1.4) 
Let $X,Y \subset \R^n$ be subanalytic sets and let $f: X \to \R^n$ be a subanalytic map.
\begin{itemize}
\item[$(i)$] $X \setminus Y$ is subanalytic.
\item[$(ii)$] If $f^{-1}(B)$ is bounded for any bounded set $B$ in $\R^n$ then $f(X)$ is subanalytic. 
\item[$(iii)$] If $f(X \cap B)$ is bounded for any bounded set $B$ in $\R^n$ then $f^{-1}(Y)$ is subanalytic.
\end{itemize}
\end{prop}

The following proposition gives a subanalytic version of the curve selection lemma:

\begin{prop} \label{lema_selecao} (\cite{Shiota}, Prop. I.2.1.7) 
Let $X \subset M$ be a subanalytic set. For each point $x \in \overline{X} \setminus X$, there exists a subanalytic map $\phi: [0,1] \longrightarrow \overline{X}$ such that $\phi \big( (0, 1] \big)  \subset X$, $\phi(0) = x$ and $\phi_{| (0,1]}$ is an analytic embedding.
\end{prop}

Any subanalytic set $X \subset \R^n$ has a locally conical structure. Precisely, we have:

\begin{prop} \label{estrutura_conica} (\cite{Coste} Thm 4.10) 
Any subanalytic set $X \subset \R^n$ has a locally conical structure, i.e., for any $x \in X$ there exists a sufficiently small ball $\B_\e(x)$ around $x$ in $\R^n$ such that $X \cap \B_\e(x)$ is homeomorphic to the cone over $X \cap \BS_\e(x)$, where $\BS_\e(x)$ is the boundary of $\B_\e(x)$.
\end{prop}

We say that a Whitney stratification that satisfies the $(w)$-condition of Kuo (see 1.3 of \cite{Verdier}, for instance) is a {\it strong Whitney stratification}. 

Recall that a real analytic set $Y$ is said {\it countable at infinity} if there exists a sequence $\{K_n\}$ of compact subsets of $Y$ such that $K_n \subset K_{n+1}$ and $Y= \cup_n K_n$.

We have:

\begin{prop} \label{estratificacao} (\cite{Verdier} Thm 2.2) 
Let $Y \subset \R^n$ be a real analytic set countable at infinity. Let $X_\beta$ be a locally finite family of subsets of $Y$ subanalytic in $Y$. Then there exists a strong Whitney stratification $\{ \mathcal{S}_\a \}$ of $Y$ such that each $X_\beta$ is a union of strata. 
\end{prop}

In particular, the proposition above gives a strong Whitney stratification for any subanalytic set $X \subset \R^n$.

\medskip
Given a subanalytic set $X \subset \R^n$ endowed with a strong Whitney stratification $\{ \mathcal{S}_\a \}$, we say that a stratified vector field $v$ on $X$ (i.e., a vector field on $X$ such that $v(x) \in T_x \mathcal{S}_{\a(x)}$ for each $x \in X$, where $\mathcal{S}_{\a(x)}$ denotes the stratum that contains $x$) is {\it rugose} if for any $p \in X$ there exists a neighborhood $W_p$ of $p$ in $\mathbb{R}^n$ and a constant $C_p > 0$, such that
$$||v(y) - v(x)|| \leq C_p ||y - x||,$$
for every $y \in W_p \cap S_{\alpha(p)}$ and every $x \in W_p \cap S_\beta$, whenever $\overline{S_\beta} \supset S_{\a(p)}$, where $S_{\a(p)}$ denotes the stratum that contains the point $p$. 

In Proposition 4.8 of \cite{Verdier}, Verdier proved that any rugose stratified vector field $v$ on a closed subanalytic set $X \subset \mathbb{R}^n$ is integrable. 

\medskip
We say that a continuous map $f: X \to W$ between subanalytic sets is {\it transversal} to a strong Whitney stratification $\mathcal{S} = \{\mathcal{S}_\a\}$ of $X$ if there exists a strong Whitney stratification $\mathcal{W} = \{\mathcal{W}_\b\}$ of $W$ such that for each $\a$ one has that $f(\mathcal{S}_\a) \subset \mathcal{W}_\b$ for some $\b$ and the restriction $f_{| \mathcal{S}_\a}: \mathcal{S}_\a \to \mathcal{W}_\b$ is a $\mathcal{C}^\infty$-submersion.

We have:

\begin{prop}(\cite{Verdier}, Prop. 4.6)\label{levantamentodecampor}
Let $Y \subset \R^n$ be a real analytic space and let $X$ be a locally closed subset of $Y$ endowed with a strong Whitney stratification $\mathcal{S} = \{\mathcal{S}_\a\}$. Let $Z$ be a smooth real analytic space and let $f: Y \longrightarrow Z$ be a continuous map transversal to $\mathcal{S}$. Let $\vec v$ be a $\mathcal{C}^\infty$ vector field on $Z$. Then there exists a rugose stratified vector field $\vec w$ on $X$ that lifts $\vec v$, i.e., for each $x \in X$ one has that $df(\vec w(x)) = \vec v(f(x))$.
\end{prop}

\begin{prop}(\cite{Verdier}, Thm. 3.3)\label{teosubmersaodosestratos}
Let $Y$ and $Z$ be real analytic sets countable at infinity and let $f: Y \to Z$ be a continuous map. Let $X$ be a locally closed subset of $Y$ endowed with a strong Whitney stratification $\mathcal{S} = \{\mathcal{S}_\a\}$. Suppose that the restriction $f_{|X}: X \longrightarrow Z$ is proper. Then there exists an open set $U$ in $f(A)$, dense in $f(A)$, which is a smooth subspace of $Z$ and subanalytic in $Z$, such that the restriction $f_{|U}: f^{-1}(U)\longrightarrow U$ is transversal to $\mathcal{S} \cap f^{-1}(U)$.
\end{prop}

\medskip
\section{Milnor-L\^e type fibrations for subanalytic maps}  
\medskip

Let $X \subset \R^m$ and $Y \subset \R^n$ be subanalytic sets, with $m \geq n$. Let: 
$$f :X \to Y$$
be a subanalytic map that extends to a continuous map $\tilde{f}: \R^m \to \R^n$.

We have:

\begin{prop} \label{prop_1}
Let $B \subset \R^m$ be a compact subanalytic set and let $\CC \subset \R^n$ be a one-dimensional subanalytic set contained in $f(X \cap B)$. For any $y \in \CC$ there exists a positive real number $\delta>0$ such that the restriction:
$$f_|: f^{-1} \big( (\CC \backslash \{y\}) \cap \BD_\delta(y) \big) \cap B \to (\CC \backslash \{y\}) \cap \BD_\delta(y)$$
is the projection of a topological trivial fibration, where $\BD_\delta(y)$ denotes the closed ball of radius $\delta$ around $y$ in $\R^n$.
\end{prop}

\begin{proof}
To simplify the notation, set $\CC^* := \CC \backslash \{y\}$ and $\mathring{\BD}_\eta := \mathring{\BD}_\eta(y)$, the open ball of radius $\eta$ around $y$ in $\R^n$. By $(i)$ of Proposition \ref{propriedades}, we have that $\CC^*$ is subanalytic in $\R^n$ so $\CC^* \cap \mathring{\BD}_\eta$ is subanalytic in $\R^n$, for any $\eta>0$.

Notice that $X \cap B$ is subanalytic in $\R^n$ and that $f(X \cap B)$ is bounded, since $B$ is compact and $f$ extends to a continuous map $\tilde{f}: \R^m \to \R^n$. In fact, $\tilde{f}(B)$ is bounded and $f(X\cap B) \subset  \tilde{f}(B)$.

We can embed $\CC$ into $\R^m$ by means of the canonical inclusion map $i(x_1, \dots, x_n) := (x_1, \dots, x_n, 0, \dots, 0)$. It is easy see that $\tilde{\CC}:= i(\CC)$ is subanalytic in $\R^m$. The same is true for $\tilde{D} := i (\mathring{\BD}_\eta)$. Moreover, we can extend the map $f: X \to Y$ to a subanalytic map $h: X \to \R^m$ by setting $h(x) := \big( f(x), 0 \big) \in \R^n \times \R^{m-n}$. Note that $h^{-1} \big( \tilde{\CC} \backslash \{(y,0)\} \cap \tilde{D} \big) = f^{-1}(\CC^* \cap \mathring{\BD}_\eta)$ and that $h(X \cap B)$ is bounded. So $(iii)$ of Proposition \ref{propriedades} gives that $f^{-1}(\CC^* \cap \mathring{\BD}_\eta)$ is subanalytic in $\R^m$, and hence $f^{-1}(\CC^* \cap \mathring{\BD}_\eta) \cap B$ is subanalytic in $\R^m$. 

So it follows from Proposition \ref{estratificacao} that there is a strong Whitney stratification $\mathcal{S} = \{\mathcal{S}_\a\}$ of $f^{-1}(\CC^* \cap \BD_\eta) \cap B$, for any $\eta>0$.

Notice that $f_{|}: f^{-1}(\CC^* \cap \mathring{\BD}_\eta) \cap B \longrightarrow \CC^* \cap \mathring{\BD}_\eta$ is proper. Moreover, by Proposition \ref{estratificacao} we can take $\eta>0$ sufficiently small such that $\CC^* \cap \mathring{\BD}_\eta$ is analytic. So Proposition \ref{teosubmersaodosestratos} gives that there exists an open set $U$ in $\CC^* \cap \mathring{\BD}_\eta$, dense in $\CC^* \cap \mathring{\BD}_\eta$, which is a smooth subspace of $\R^n$ and subanalytic in $\R^n$, such that the restriction $f_{|U}: f^{-1}(U) \cap B \longrightarrow U$ is transversal to $\mathcal{S} \cap f^{-1}(U) \cap B$.

Now, since $y \in \overline{U} \backslash U$, it follows from Proposition \ref{lema_selecao} that there exists a subanalytic map $\phi: [0,1] \longrightarrow \overline{U}$ such that $\phi \big( (0, 1] \big)  \subset U$, $\phi(0) = y$ and $\phi_{| (0,1]}$ is an analytic embedding. Hence we can take $\eta>0$ sufficiently small such that $U = \CC^* \cap \mathring{\BD}_\eta$, which is a smooth path in $\R^n$. 

Therefore the restriction $f_{|}: f^{-1}(\CC^* \cap \mathring{\BD}_\eta) \cap B \longrightarrow \CC^* \cap \mathring{\BD}_\eta$ is transversal to $\mathcal{S} \cap f^{-1}(\CC^* \cap \BD_\eta) \cap B$. So now we can finally prove that it is a topological locally trivial fibration. 

Let $\vec{v}$ be a smooth vector field in $\CC^* \cap \mathring{\BD}_\eta$ that goes to zero in a finite time. By Proposition \ref{levantamentodecampor} we can to lift $\vec v$ to a rugose (and hence integrable) stratified vector field $\vec w$ in $f^{-1}(\CC^* \cap \mathring{\BD}_\eta) \cap B$ that gives the trivialization of $f_{|}: f^{-1}(\CC^* \cap \mathring{\BD}_\eta) \cap B \longrightarrow \CC^* \cap \mathring{\BD}_\eta$. Then the theorem follows taking $\delta<\eta$.
\end{proof}

\begin{example} 
Consider the continuous subanalytic map $f: \R^3 \to \R^2$ given by:
$$
f(x,y,z) = \left\{
\begin{array}{lcl}
\big( (x+y)^2, 2(x^2+y^2+z^3) \big) & \mbox{if} & (x+y) \geq 0\\
\big( -(x+y)^2, 2(x^2+y^2+z^3) \big) & \mbox{if} & (x+y) < 0
\end{array}
\right.
\, .$$
By Proposition \ref{prop_1}, for any compact set $B \subset \R^m$, for any one-dimensional subanalytic curve $\CC \subset \R^2$ and for any $y \in \CC$, there exists $\delta>0$ sufficiently small such that the restriction:
$$f_|: f^{-1} \big( (\CC \backslash \{y\}) \cap \BD_\delta(y) \big) \cap B \to (\CC \backslash \{y\}) \cap \BD_\delta(y)$$
is the projection of a topological locally trivial fibration. One can easily check that $f^{-1}(t_1,t_2)$ is homeomorphic to $\R$, for any $(t_1,t_2) \in \R^2$.
\end{example}

We also suggest the reader to see the smooth non-analytic map $\Psi: \R^m \to \R^2$ gave in Section 5.2 of \cite{CMSS}. 

\medskip
Now let us consider the higher dimensional case. We have:

\begin{prop} \label{prop_2}
Let $f: X \to Y$ be as above and let $B \subset \R^m$ be a compact subanalytic set and let $Z \subset \R^n$ be a subanalytic set contained in $f(X \cap B)$. There exists a subset $W \subset Z$, subanalytic in $\R^n$, with $\dim W< \dim Z$, such that the restriction:
$$f_|: f^{-1}(Z \backslash W) \cap B \to Z \backslash W$$
is the projection of a topological locally trivial fibration. 
\end{prop}

\begin{proof}
Consider a strong Whitney stratification $\mathcal{Z} = \{ \mathcal{Z}_\b \}$ of $Z$ and let $\tilde{W}$ be the union of all the strata $\mathcal{Z}_\b$ such that $\dim \mathcal{Z}_\b < \dim Z$. Then $Z \backslash \tilde{W}$ is real analytic, smooth and subanalytic in $Z$.

Following as in the first three paragraphs of the proof of Proposition \ref{prop_1}, one can see that $f^{-1}(Z \backslash \tilde{W}) \cap B$ is subanalytic in $\R^m$. So it follows from Proposition \ref{estratificacao} that $f^{-1}(Z \backslash \tilde{W})  \cap B$ admits a strong Whitney stratification $\mathcal{S} = \{\mathcal{S}_\a\}$.

So applying Proposition \ref{teosubmersaodosestratos} to the restriction $f_|: f^{-1}(Z \backslash \tilde{W}) \cap B \to Z \backslash \tilde{W}$ we obtain an open set $U$ in $Z \backslash \tilde{W}$, dense in $Z \backslash \tilde{W}$, which is a smooth subspace of $\R^n$ and subanalytic in $\R^n$, such that the restriction $f_{|U}: f^{-1}(U) \cap B \longrightarrow U$ is transversal to $\mathcal{S} \cap f^{-1}(U) \cap B$.

Let $U^c$ be the complement of $U$ in $Z \backslash \tilde{W}$. Hence $U^c$ is subanalytic in $\R^n$. Therefore the disjoint union given by $W:= \tilde{W} \cup U^c$ is subanalytic in $\R^n$ and $\dim W< \dim Z$.

So the restriction $f_|: f^{-1}(Z \backslash W) \cap B \to Z \backslash W$ is transversal to $\mathcal{S} \cap f^{-1}(U) \cap B$. So proceeding as in the last paragraph of the proof of Proposition \ref{prop_1}, one gets that it is the projection of a topological locally trivial fibration. 
\end{proof}

\begin{remark}
Notice that if $\CC$ is a subanalytic curve and $y \in \CC$, then one can apply Proposition \ref{prop_2} to the set $Z:= \CC \cap \BD_\eta(y)$. One can check that, for $\eta>0$ sufficiently small, then either $W= \emptyset$ or $W= \{y\}$. This gives Proposition \ref{prop_1}.
\end{remark}

Applying Proposition \ref{prop_2} repeatedly, we get Theorem \ref{theo_1}.

\medskip
\section{Fibration theorems for analytic maps defined on subanalytic sets}  
\medskip

Let $X \subset \R^m$ and $Y \subset \R^n$ be subanalytic sets and let $f :X \to Y$ be a subanalytic map that extends to an analytic map $\tilde{f}: \R^m \to \R^n$.

Let $B \subset \R^m$ be a compact subanalytic set and define the restriction:
$$f_B: X \cap B \to Y \, .$$

Let $\mathcal{Y}= \{ \mathcal{Y}_\b \}_{\b \in \Omega}$ be a strong Whitney stratification of $Y$. We can consider a strong Whitney stratification $\mathcal{S}= \{ \mathcal{S}_\a \}_{\a \in \Gamma}$ of $X \cap B$ such that for any $\a \in \Gamma$ one has that $f_B(\mathcal{S}_\a) \subset \mathcal{Y}_\b$, for some $\b = \b(\a) \in \Omega$. 

Set $k:= \dim Y$ and for each $i=1, \dots, k$ set: 
$$\Omega_i := \{\b \in \Omega; \ \dim \mathcal{Y}_\b = i \}  \, .$$
and
$$Y_i := \cup_{\b \in \Omega_i} \mathcal{Y}_\b  \, .$$
Since the stratification $\mathcal{Y}$ satisfies the boundary condition, one has that each $Y_i$ is a smooth submanifold of $\R^n$.

Now, for each $i=1, \dots, k$ set: 
$$\Gamma_i := \{\a \in \Gamma; \  \b(\a) \in \Omega_i \}  \, .$$
So $\a \in \Gamma_i \Leftrightarrow f(\mathcal{S}_\a) \subset Y_i$. Then set:
$$X_i := \cup_{\a \in \Gamma_i} \mathcal{S}_\a  \, .$$
So $f(X_i) \subset Y_i$. Also set:
$$\Sigma_i:= \bigcup_{\a \in \Gamma_i} Crit(f_{| \mathcal{S}_\a})  \, ,$$
where $Crit(f_{| \mathcal{S}_a})$ denotes the set of points of $\mathcal{S}_\a$ where the corres\-ponding restriction $f_{| \mathcal{S}_\a}: \mathcal{S}_\a \to \mathcal{Y}_{\b(\a)}$ fails to be a submersion. Notice that $\Sigma_i \subset X_i$ is analytic. Moreover, for each $i=1, \dots, k$ set: 
$$\Delta_i := f_B(\Sigma_i) \subset Y_i \, ,$$
which is subanalytic in $\R^n$.

Finally, we have:

\begin{definition} \label{definition_1}
The singular set of $f_B$ is the set:
$$\Sigma_{f_B} := \bigcup_{i=1}^k \Sigma_i  \, ,$$
and the {\it discriminant set of $f_B$} is the set $\Delta_{f_B} := f_B(\Sigma_{f_B})$.
\end{definition}

Notice that $\Sigma_{f_B} \subset (X \cap B)$ is analytic and that $\Delta_{f_B} \subset Y$ is sub\-analytic in $\R^n$.

With the notation above, we have:

\begin{theo} \label{theo_2}
Let $f: X \to Y$ be a subanalytic map as above and let $B \subset \R^m$ be a compact subanalytic set. For each $i=1, \dots, k$ one has that:
\begin{itemize}
\item[$(i)$] $(Y_i \setminus \Delta_i) \cap f(X \cap B)$ is an open smooth submanifold of $Y_i$.
\item[$(ii)$] The restriction:
$$f_|: f^{-1}(Y_i \setminus \Delta_i) \cap B  \to  (Y_i \setminus \Delta_i) \cap f(X \cap B) $$
is the projection of a topological locally trivial fibration.
\end{itemize}
\end{theo}

\begin{proof}
Since $f_B$ is a continuous map from a compact space to a Hausdorff space, it is proper and closed. Therefore $f(X \cap B)$ is a closed subset of $Y$, which implies that $f(X \cap B) \cap Y_i$ is a closed subset of $Y_i$. On the other hand, $\Sigma_{f_B}$ is a closed subset of $X \cap B$, and hence $\Delta_{f_B}$ is closed in $Y$, which implies that $\Delta_i = \Delta_{f_B} \cap Y_i$ is closed in $Y_i$. So $Y_i \setminus \Delta_i$ is an open set of $Y_i$.

Clearly, the boundary points (as a topological space) of $f(X \cap B) \cap Y_i$ are contained in $\Delta_i$. Therefore $f(X \cap B) \cap (Y_i \setminus \Delta_i)$ is an open subset of $Y_i$. But since $Y_i$ is smooth, it follows that $f(X \cap B) \cap Y_i \setminus \Delta_i$ is an open submanifold of $Y_i$. This proves $(i)$.

One can check that $f^{-1}(Y_i \setminus \Delta_i) \cap B$ is subanalytic in $X$. Moreover, the restriction $f_|: f^{-1}(Y_i \setminus \Delta_i) \cap B  \to  (Y_i \setminus \Delta_i) \cap f(X \cap B) $ is transversal to the stratification $\mathcal{S} \cap f^{-1}(Y_i \setminus \Delta_i) \cap B$ of $f^{-1}(Y_i \setminus \Delta_i) \cap B$ induced by $\mathcal{S}$. So $(ii)$ follows from Proposition \ref{levantamentodecampor}. 
\end{proof}

Notice that $(i)$ of Theorem \ref{theo_main} follows immediately from Theorem \ref{theo_2}, putting $Y=\R^n$. 

Moreover, Theorem \ref{theo_2} together with Theorem \ref{theo_1} give:

\begin{theo} \label{theo_3}
Let $X \subset \R^m$ and $Y \subset \R^n$ be subanalytic sets and let $f :X \to Y$ be a subanalytic map that extends to an analytic map $\R^m \to \R^n$. Let $B \subset \R^m$ be a compact subanalytic set. For each $i=1, \dots, k$ such that $\dim \Delta_i >0$, there exist subsets $W_0^i, \dots, W_{l_i}^i$, with $1 \leq l_i \leq \dim \Delta_i$, such that: 
\begin{itemize} 
\item[$(i)$] $W_0^i \subset W_1^i \subset \dots \subset W_{l_i-1}^i \subset W_{l_i}^i = \Delta_i$.
\item[$(ii)$] $W_j^i$ is subanalytic in $\R^n$, for each $j=0, \dots, l_i$.
\item[$(iii)$] $\dim W_{j-1}^i < \dim W_j^i$, for each $j=1, \dots, l_i$.
\item[$(iv)$] $W_j^i \setminus W_{j-1}^i$ is a smooth subset of $\R^n$ and subanalytic in $\R^n$, for each $j=1, \dots, l_i$.
\item[$(v)$] The restriction:
$$f_|: f^{-1}(W_j^i \backslash W_{j-1}^i) \cap B \to W_j^i \backslash W_{j-1}^i$$
is the projection of a topological locally trivial fibration, for each $j=1, \dots, l_i$. 
\end{itemize}
\end{theo}

\medskip
Putting $Y=\R^n$ one gets $(ii)$ of Theorem \ref{theo_main}.

\medskip
\begin{example}
Let $\tilde{f}: \R^4 \to \R^2$ be the real analytic map given by: 
$$\tilde{f}(x,y,z,w) := \big( (x+y)^2, 2(x^2+y^2+z^2 + w^3) \big) \, .$$
Let $X \subset \R^4$ be the analytic set given by $X=\{x^2-y^2=0\}$ and consider the analytic map $f:X \to \R^2$ given by the restriction of $\tilde{f}$. 

The singular set of $X$ is the set $\Sigma(X):= \{x=y=0\}$ and we have a strong Whitney stratification of $X$ given by the strata $\mathcal{S}_1 := X \backslash \Sigma(X)$ and $\mathcal{S}_2 := \Sigma(X)$. An easy calculation shows that: 
$$\Sigma_f = \{x+y=0\} \cup \{x-y=z=w=0\} \ ,$$
so the discriminant of $f$ is given by:
$$\Delta_f = \{u=0\} \cup \{u-v=0; u \geq 0\} \, .$$

Theorem \ref{theo_main} gives that for any $\e>0$ there exists $\eta>0$ sufficiently small such that the restrictions:
$$f_|: f^{-1}(\R^2 \setminus \Delta_f) \cap \B_\e  \to  (\R^2 \setminus \Delta_f) \cap f(X \cap \B_\e) $$
and
$$f_|: f^{-1} \big( (\Delta_f \backslash \{0\}) \cap \BD_\eta \big) \cap \B_\e \to (\Delta_f \backslash \{0\}) \cap \BD_\eta$$
are projections of topological locally trivial fibrations.

It is easy to see that, for any $\e>0$, one has that: 
\begin{itemize}
\item[$(i)$] If $t_1>0$ and $t_2 \neq t_1$ then $f^{-1}(t_1,t_2) \cap \B_\e$ is a smooth manifold homeomorphic to two disjoint copies a closed interval;
\item[$(ii)$] If $t_1>0$ and $t_2=t_1$ then $f^{-1}(t_1,t_2)$ is the singular surface given by $\{x-y=0\} \cap \{z^2 + w^3=0\}$, which is homeomorphic to a closed disk $\D^2$; 
\item[$(iii)$] If $t_1=0$ and $t_2 \neq 0$ then $f^{-1}(t_1,t_2)$ is the non-singular surface given by $\{x+y=0\} \cap \{4x^2 +2z^2 + 2w^3= t_2\}$, which is homeomorphic to either $\D^2 \backslash \{0\}$ or to two disjoint copies of $\D^2$;
\item[$(iv)$] If $t_1=t_2=0$ then $f^{-1}(t_1,t_2)$ is the singular surface given by $\{x+y=0\} \cap \{4x^2 +2z^2 + 2w^3= 0\}$, which is homeomorphic to the cone over two circles.
\end{itemize}
\end{example}

Next, we give an example where it is possible to explicit the set $W$ of Theorem \ref{theo_main}.

\medskip
\begin{example}
Let $f: \R^m \to \R^n$ be a real analytic map given by: 
$$f(x_1, \dots, x_m) := \big( \sum_{j=1}^m a_{1j} x_j^p, \dots, \sum_{j=1}^m a_{ij} x_j^p, \dots, \sum_{j=1}^m a_{nj} x_j^p  \big)  \, ,$$
with $m \geq n \geq 2$ and $p\geq 2$ integer. Suppose that the coefficients $a_{ij}$ are generic enough (which means that any collection of $n$-many vectors $\vec{a}_j := (a_{1j}, \dots, a_{nj})$, with $j \in \{1, \dots, m\}$, is linearly independent, and that $0 \in \R^n$ is in the convex hull of the vectors $\vec{a}_j$).

Such family is a generalization of the homogeneous quadratic mappings $\R^m \to \R^2$ studied by S.L. de Medrano in \cite{Santiago}. In that paper, de Medrano considered the case $n=p=2$. He showed that there is a local Milnor-L\^e type fibration outside the discriminant of $f$, which is the union of the $m$-many line-segments through the origin and the points $(a_{1j}, a_{2j})$ in $\R^2$. Moreover, he completely described the topology of the associated fiber in terms of the configuration of the points $(a_{1j}, a_{2j})$ in $\R^2$.

In our general situation, the critical set $Crit(f)$ of $f$ is given by the union of the $(n-1)$-dimensional planes:
$$\Sigma_{j_1, \dots, j_k, \dots, j_{n-1}} := \{x_j=0, \ \forall j=1, \dots, m \ \text{with} \ j \neq j_k , \ k=1, \dots, n-1 \} \, ,$$
with $j_1, \dots, j_{n-1}$ distinct indices in $\{1, \dots, m\}$. These sets consist of the points of $\R^m$ where the Jacobian matrix of $f$ has rank less than $n$.

We can also set:
$$\widetilde{\Sigma}_{j_1, \dots, j_k, \dots, j_{n}} := \{x_j=0, \ \forall j=1, \dots, m \ \text{with} \ j \neq j_k , \ k=1, \dots, n \} \, ,$$
which are $(n-2)$-dimensional planes formed by the points of $\R^m$ where the Jacobian matrix of $f$ has rank less than $n-1$.

One can check that for any sufficiently small closed ball $\B_\e$ around $0$ in $\R^m$ there exists a sufficiently small open ball $\mathring{\B}_\delta$ around $0$ in $\R^n$ such that the singular set $\Sigma_{f_B}$ of the restriction: 
$$f_B: f^{-1}(\mathring{\B}_\delta) \cap \B_\e \to \mathring{\B}_\delta$$ 
of $f$ equals $Crit(f) \cap \B_\e$.

So the discriminant set $\Delta_{f_B} \subset \mathring{\B}_\delta$ is given by the union of the following parametrized hyperplanes $\Delta_{j_1, \dots, j_k, \dots, j_{n-1}}$ in $\R^n$:
$$(t_1, \dots, t_{n-1}) \mapsto \big( a_{1j_1} t_1^p + \dots + a_{1j_{n-1}} t_{n-1}^p, \dots, a_{nj_1} t_1^p + \dots + a_{nj_{n-1}} t_{n-1}^p  \big)   \, .$$

Since $\Delta_{f_B}$ is a finite union of hyperplanes in $\R^n$, its singular set is a finite union of $(n-2)$-dimensional planes, which we will denote by $\Sigma_\Delta$. Also notice that $f(\widetilde{\Sigma}_{j_1, \dots, j_k, \dots, j_{n}})$ is contained in $\Sigma_\Delta$. So $f^{-1}(\Delta_{f_B} \backslash \Sigma_\Delta) \cap \mathring{\B}_\e$ is a smooth manifold and the restriction of $f$ to $f^{-1}(\Delta_{f_B} \backslash \Sigma_\Delta) \cap \mathring{\B}_\e$ is a submersion. Hence the restriction: 
$$f_|: f^{-1}(\Delta_{f_B} \backslash \Sigma_\Delta) \cap \B_\e \to \Delta_{f_B} \backslash \Sigma_\Delta$$ 
is the projection of a topological locally trivial fibration.
\end{example}


\end{document}